\documentclass[12pt]{amsart}   \usepackage{amssymb,latexsym}
\usepackage{amsfonts,mathrsfs}
\usepackage[margin=1.4in]{geometry}

\newtheorem{theorem}{Theorem}[section]
\newtheorem{lemma}[theorem]{Lemma}

\newtheorem{corollary}[theorem]{Corollary}
\newtheorem{definition}[theorem]{Definition}

\theoremstyle{remark}

\newcommand{\tr}{\mbox{\rm tr\,}}
\newcommand{\udots}
{\mathinner{\mskip1mu\raise1pt\vbox{\kern7pt\hbox{.}}  
\mskip2mu\raise4pt\hbox{.}\mskip2mu\raise7pt\hbox{.}\mskip1mu}} 
\begin{document}
\title{gradient estimates via two-point  function  under Ricci flow }
\thanks{\noindent \textbf{MR(2010)Subject Classification} 53C44}

\author{min chen}
\address[Corresponding author] {University of Science and Technology of China, No.96, JinZhai Road Baohe District,Hefei,Anhui, 230026,P.R.China.} 
\email{cmcm@mail.ustc.edu.cn}

\thanks{The research is supported by the National Nature Science Foudation of China  No. 11721101 No. 11526212 }
\begin{abstract}
 We derive estimates relating the values of a solution at any two points to the distance between the points, for quasilinear parabolic equations on compact Riemannian manifolds under the Ricci flow.  
\end{abstract}
\maketitle

\numberwithin{equation}{section}
\section{Introduction}
 Andrews and Clutterbuck \cite{2,3} and Andrews \cite{4} study the two-point estimates and their applications in a variety of geometric contexts. Recently,
Ben Andrew and Changwei Xiong \cite{1} use the two-point estimates  to deduce  gradient estimates of the solutions of the following quasilinear equations.
\begin{equation}
\left[ \alpha(u,|Du|)\frac{D_iuD_ju}{|Du|^2}+\beta(u,|Du|)(\delta_{ij}-\frac{D_iuD_ju}{|Du|^2})\right]D_iD_ju+q(u,|Du|)=0,
\end{equation}
where the left side of $(1.1)$ is continuous on $ \mathbb R\times TM^*_x\times \mathcal{L}^2_s(TM)$, $\alpha$ and $\beta$ are nonnegative functions and $\beta(s,t)>0$ for $t>0.$ They prove  the following result in their recent work \cite{1}.
\begin{theorem}\cite{1} 
Let $M^n$ be a compact manifold with $Ric\geq 0$ and $u$ be a viscosity solution of Equation $(1.1)$. Suppose the barrier $\varphi: [a, b]\rightarrow [\inf u,\sup u]$ satisfies
\begin{equation}
\varphi'>0,
\end{equation}
\begin{equation}
\frac{d}{dz} \Bigg ( \frac{\varphi''\alpha(\varphi, \varphi')+q(\varphi,\varphi')}{\varphi'\beta(\varphi,\varphi')} \Bigg ) <0. 
\end{equation}
Moreover let $\psi$ be the inverse of $\varphi$, i.e., $\psi (\varphi(z))=z.$ Then it holds that  
\[\psi(u(y))-\psi(u(x))-d(x,y)\leq 0, \forall x,y \in M.\]
\end{theorem}
By allowing $y$ to approach $x$, one gets the following gradient estimate
\begin{corollary}
Under the conditions of Theorem 1.1, for every $u\in C^1(M),$  then it holds that 
\[\nabla u(x)\leq \varphi'(\psi(u(x))),\]
for every $x\in M.$ \
\end{corollary}
We know that  Modica-type gradient estimates obtained by using the P-function in \cite{5} is a special case of this result with $\alpha=\beta=1$ and a generalized Modica's result in \cite{6} is also a special case of this result with $\alpha = 2\Psi''(z)z+\Psi'(z)$ and $\beta=\Psi(z)$. The details of the two-point estimate are comparatively simple and geometric compared to the calculations involved in the $P$- function approach. Since this method does not involve differentiating the equation, and consequently applies with minimal regularity requirements on the solution $u,$ corresponding to the viscosity solution requirement.

 We wonder if  this method can be use to derive the gradient estimates  of the parabolic equations. Azagra, Jim$\acute{e}$nez-Sevilla and Maci$\grave{a}$ \cite {9} define the viscosity solution and prove the maximum principle for semicontinuous function in the parabolic version on manifolds. What's more, various authors consider the gradient estimates under the Ricci flow,
\[\frac{\partial}{\partial t}g(x,t)=-2Ric(x,t).\]Shiping, Liu \cite{7} derived gradient estimates on a closed Riemannian manifold. 
\begin{theorem} \cite{7}
Let $(M,g(t))$ be a closed Riemannian manifold, where $g(t)$ evolves by Ricci flow in such a way that $-K_0 \leq Ric\leq K_1$ for $t\in[0,T].$ If $u$ is a positive solution to the equation $(\Delta-\partial_t)u(x,t)=0,$ then for $(x,t)\in M \times (0,T],$ it holds that 
\[\frac{|\nabla u(x,t)|^2}{u^2(x,t)}-\alpha \frac{u_t(x,t)}{u(x,t)}\leq \frac{n\alpha^2}{t}+\frac{n\alpha^3K_0}{\alpha-1}+n^{\frac{3}{2}}\alpha^2(K_0+K_1),\]
for any $\alpha>1.$ 
\end{theorem}
In fact, we find the argument also works when the metric  evolves as a supersolution of the Ricci flow, i.e., $\frac{\partial g}{\partial t}\geq -2Ric.$ We prove the following result.
\begin{theorem} \label{thm:1.1} 
Let $M^n$ be a compact manifold with diameter $D$ and $g(t)$ a time-dependent metric on $M$ satisfying $\frac{\partial g}{\partial t}\geq -2 Ric.$  Assume Ricci curvature satisfies $Ric\ge 0$ for $t\in [0,T)$ and $u:M\times [0,T)\rightarrow \mathbb{R}$ is a viscosity solution of the heat equation
 \begin{equation}
u_t=\left[\alpha(u,|Du|,t)\frac{D_iuD_ju}{|Du|^2}+\beta(t)(\delta_{ij}-\frac{D_iuD_ju}{|Du|^2})\right]D_iD_ju+q(u,|Du|,t),
\end{equation}
where  $\beta(t)\ge 1.$ Suppose $\varphi:[0,D]\times [0,T]\rightarrow \mathbb{R}$  satisfies
\begin{align*}
 \frac{d}{ds}\Bigg ( \frac{\varphi_t-\varphi''\alpha(\varphi, \varphi',t)+q(\varphi,\varphi',t)}{\varphi'\beta(t)} \Bigg )& <0,\\
\varphi'>0.
\end{align*}
Assume the range of $u(\cdot,0)$ is contained in the interval $[\varphi(0,0), \varphi(D,0)].$
Let $\Psi$ be given by inverting $\varphi$ for each $t$ so that $\varphi(\Psi(z,t),t)=z$ for each $z$ and $t$. Assume that for all $x$ and $y$ in $M,$ 
\[\Psi(u(y,0),0)-\Psi(u(x,0),0)-d_0(x,y)\leq 0.\]
Then 
\[\Psi(u(y,t),t)-\Psi(u(x,t),t)-d_t(x,y)\leq 0\]
for all $x,y\in M$ and $t\in [0,T).$
\end{theorem}
Then we have  the following gradient estimate immediately.
\begin{corollary}
Under the conditions of Theorem 1.4, for every $x\in M$ and $t\geq 0,$ we have
\[|\nabla u(x,t)|\leq \varphi'(\Psi(u(x,t),t),t).\]
\end{corollary}
We know that the graphical mean curvature flow and the Laplacian heat flow are two important examples of the heat equation of Theorem 1.4. Let us consider another type of heat equation which contains important examples such as p-Laplace heat flow.\begin{theorem} \label{thm:1.2} 
Let $M^n$ be a compact manifold  with diameter $D$ and $g(t)$ a time-dependent metric on $M$ satisfying $\frac{\partial g}{\partial t}\geq -2 Ric.$ Assume Ricci curvature satisfies $|Ric|\le \kappa$ for $t\in [0,T)$ and $u:M\times [0,T)\rightarrow \mathbb{R}$ is a viscosity solution of the heat equation
 \begin{equation}
u_t=\left[\alpha(|Du|,t)\frac{D_iuD_ju}{|Du|^2}+\beta(|Du|,t)(\delta_{ij}-\frac{D_iuD_ju}{|Du|^2})\right]D_iD_ju+q(|Du|,t),
\end{equation}

where  $\alpha \geq 0,$  $\beta>0,$  Suppose $\varphi:[0,D]\times [0,T)\rightarrow \mathbb{R}$  satisfies
\begin{equation}
 \varphi_t\geq \varphi''\alpha( \varphi',t)+\kappa s|\varphi'(1-\beta(\varphi',t))|.
\end{equation}
Assume that for all $x$ and $y$ in $M,$ 
\[u(y,0)-u(x,0)-2\varphi(\frac{d_0(x,y)}{2},0) \leq 0.\]
Then 
\[u(y,t)-u(x,t)-2\varphi(\frac{d_t(x,y)}{2},t)\leq 0.\]
\end{theorem}
\begin{corollary}
Under the conditions of Theorem 1.6, for every $x\in M$ and $t\geq 0,$ we have
\[|\nabla u(x,t)|\leq \varphi'(0,t).\]
\end{corollary}
\setcounter{equation}{0}
 \hspace{0.4cm}

\section{prelimilaries}
\setcounter{equation}{0}
\medskip
\begin{definition} (\cite{8}) Let $f:(0,T)\times M \rightarrow (-\infty, +\infty)$ be a lower semi-continuous $(LSC)$ function. The parabolic second order subjet of $f$ at a point $(t_0,x_0)\in(0,T)\times M$ is defined by \\
\begin{align*}
\mathcal P^{2,-}f(t_0,x_0):=&\{(D_t\varphi(t_0,x_0),D_x\varphi(t_0,x_0),D^2_x\varphi(t_0,x_0)):\varphi \text{ is once continuously}\\ & \text{differantiable in}\; t\in (0,T), \text{twice continuously differentiable in}\; x\in M \\&\text{and} \;f-\varphi \;\text{attains a local minimum at}\; (t_0,x_0)\}.\\
\end{align*}
Similarly, for an upper semi-continuous  $(USC)$ function $f:(0,T)\times M \rightarrow [-\infty,+\infty),$ The parabolic second order superjet of $f$ at $(t_0,x_0)$ is defined by 
\begin{align*}
\mathcal P^{2,-}f(t_0,x_0):=&\{(D_t\varphi(t_0,x_0),D_x\varphi(t_0,x_0),D^2_x\varphi(t_0,x_0)):\varphi \text{ is once continuously}\\ & \text{differantiable in}\; t\in (0,T), \text{twice continuously differentiable in}\; x\in M \\&\text{and} \;f-\varphi \;\text{attains a local maximum at}\; (t_0,x_0)\}.\\
\end{align*}
\end{definition}

\begin{definition}(\cite{8})
Let $f:(0,T)\times M \rightarrow(-\infty, +\infty]$ be a LSC function and $(t,x)\in (0,T)\times M.$  $\bar{\mathcal P}^{2,-}f(t_0,x_0)$ is defined to be the set of the $(a,\zeta,A)\in \mathbb R\times TM^
*_x\times \mathcal{L}^2_s(TM_x)$ such that there exist a sequence $(x_k,a_k,\zeta_k,A_k)$ in $M\times \mathbb{R}\times TM^*_{x_k}\times \mathcal{L}^2_s(TM_{x_k})$ satisfying:
\begin{align*}
&i)\; (a_k,\zeta_k,A_k)\in P^{2,-}f(t_k,x_k),\\
&ii) \; \lim_k \;(t_k,x_k,f(t_k,x_k),a_k,\zeta_k,A_k)=(t,x,f(t,x),a,\zeta,A).
\end{align*}
The corresponding definition of $\bar{P}^{2,+}f(t_0,x_0)$ when $f$ is an upper semicontinuous function is then clear.
\end{definition}
\begin{theorem}(\cite{8}) 
Let $M_1,\cdots,M_k$ be Riemannian manifolds, and $\Omega_i\in M_i$ open subsets. Define $\Omega= (0,T)\times \Omega_1\times \cdots \times \Omega_k$. Let $u_i$ be upper semicontinuous functions on $(0,T)\times \Omega_i, i=1,\cdots,k;$ let $\varphi $ be a function defined on $\Omega$ such that  it is once continuously differentiable in $t\in (0,T)$ and twice continuously differentiable in $x:=(x_1,\cdots,x_k)\in \Omega_1\times \cdots \times \Omega_k$ and set 
\[ w(t,x_1,\cdots, x_k)=u_1(t,x_1)+\cdots+u_k(t,x_k)\]
for $(t,x_1,\cdots\ x_k)\in \Omega$. Assume that $(\hat{t}, \hat{x_1},\cdots,\hat{x_k})$ is maximum of $\omega-\varphi$ in $\Omega.$ Assume, moreover, that there is an $\tau>0$ such that for every $M>0$ there is $C>0$ such that for $i=1,\cdots,k,$
\begin{equation}
\left\{ \begin{aligned} 
&a_i\leq C \; \text{whenever}\;(a_i,\zeta_i,A_i)\in \bar{P}^{2,+}_{M_i}u_i(t,x_i) \\
&d(x_i,\hat{x_i})+|t-\hat{t}|\leq \tau\; \text{and}  \; |u_i(t,x_i)|+|\zeta_i|+||A_i||\leq M.\\
\end{aligned}
\right.
\end{equation}
Then, for each $\epsilon>0$ there exist real numbers $b_i$ and bilinear forms $B_i\in \mathcal{L}^2_s(TM_{i})_{\hat{x}_i},i=1,\cdots,k$, such that 
\[(b_i,D_{x_i}\varphi(\hat{t},\hat{x}_1,\cdots,\hat{x}_k),B_i)\in (a_i,\zeta_i,A_i)\in \bar{P}^{2,+}_{M_i}u_i(\hat{t},\hat{x}_i)\]
for $i=1,\cdots,k,$ and the block diagonal matrix with entries $B_i$ satisfies 
\begin {equation}
-(\frac{1}{\epsilon}+||A||) I \leq {\left( \begin{array}{ccc}
B_1 & \cdots & 0\\
\vdots& \ddots & \vdots \\
0&\cdots &B_k
\end{array}
\right )} \leq A+\epsilon A^2,
\end{equation}
where $A=D^2_x\varphi (\hat{t},\hat{x}_1,\cdots,\hat{x}_k)$ and $b_1+\cdots+b_k= \frac{\partial {\varphi}}{\partial {t}}(\hat{t},\hat{x}_1,\cdots,\hat{x}_k).$\\
\end{theorem}
Using the same method as in the proof of the Lemma 8 in [2], we have
\begin{lemma}
Let $u$ be a continuous function and $\varphi:\mathbb{R}\times [0,T)\rightarrow \mathbb{R}$ be a $C^{2,1}$ function with $\varphi \geq0.$ Let $\Psi:\mathbb{R}\times[0,T) \rightarrow \mathbb{R} $ be the inverse of $\varphi$, so that 
\[\Psi(\varphi(u(y,t),t),t)=u(y,t).\]
(i) Suppose $(a,\zeta, A)\in \mathcal{P}^{2,+}(\Psi \circ u)(\hat{x},\hat{t})),$ then we have
\[(\varphi_t+\varphi',\varphi'\zeta,\varphi''\zeta \otimes \zeta+\varphi'A)\in \mathcal{P}^{2,+}u(\hat{x},\hat{t}),\]
where all derivatives of $\varphi$ are evaluated at $\Psi(u(\hat{x},\hat{t}),\hat{t}).$\\
(ii) Suppose$(a,\zeta,A)\in \mathcal{P}^{2,-}(\Psi \circ u)(\hat{x},\hat{t}),$ then we have 
\[(\varphi_t+\varphi',\varphi'\zeta,\varphi''\zeta \otimes \zeta+\varphi'A)\in \mathcal{P}^{2,-}u(\hat{x},\hat{t}),\]
where all derivatives of $\varphi$ are evaluated at $\Psi(u(\hat{x},\hat{t}),\hat{t}).$\\
(iii) The same holds if we replace the semijets by their closures.
\begin{proof}
According to Definition 2.1
 \begin{align*}
\mathcal P^{2,-}u(t_0,x_0):=&\{(D_t\varphi(t_0,x_0),D_x\varphi(t_0,x_0),D^2_x\varphi(t_0,x_0):\varphi \text{ is once continuously}\\ & \text{differantiable in}\; t\in (0,T), \text{twice continuously differentiable in}\; x\in M \\&\text{and} \;u-\varphi \;\text{attains a local maximum at}\; (t_0,x_0)\}.\\
\end{align*}
Assume $(a,\zeta, A)\in \mathcal{P}^{2,+}(\Psi \circ u)(\hat{x},\hat{t}),$ let $h$ be $C^{2,1}$ function such that $\Psi(u(x,t),t)-h(x,t)$ has a local maximum at $(\hat{x},\hat{t})$ and $(h_t,Dh,D^2h)(\hat{x},\hat{t})=(a,\zeta,A).$ Since $\varphi$ is increasing, we find that  $u(x,t)-\varphi (h(x,t),t)=\varphi(\Psi(u(x,t),t),t)-\varphi(h(x,t),t)$ has a local maximum at $(\hat{x},\hat{t}).$ So it follows that 
\[(\varphi_t+\varphi'a,\varphi'\zeta,\varphi''\zeta\otimes\zeta+\varphi'A)\in \mathcal{P}^{2,+}u(\hat{x},\hat{t}).\]
(ii) can be proved by an similar argument.\\
(iii) follows by an approximation.
\end{proof}
\end{lemma}
\begin{proof}[Proof of Theorem \ref{thm:1.1}]
 Let $\epsilon >0$ be arbitrary, and consider the first time $t_0>0$ and points $x_0$ and $y_0$ in $M$ at which the inequality 
\[\Psi(u(y,t),t)-\Psi (u(x,t),t)-d_t(x,y)-\epsilon(1+t)\leq 0\]
reaches equality. Note that if $\epsilon>0,$ then we necessarily have $y_0\neq x_0$. Even though the length of the curve depends explicitly on $t$ through the time-dependence of the metric $g$, we can still replace $d_t(x,y)$ by a smooth function $\tilde d_t(x,y)$ as in the proof of Theorem 6 in [1] within a neighborhood of $(x_0,y_0) $ at any fixed time $t.$ Let $\gamma_0(s)$ be a minimizing geodesic joining $x_0$ and $y_0$ parametrized by arc length at time $t_0$, i.e., $|\gamma_0'(s)|_{g(t_0)}=1$ with length $l=\mathscr{L}_{g(t_0)}(\gamma_0)=d_{t_0}(x_0,y_0)$. Let  $\{e_i(s)\}_{i=1}^{n}$ be parallel orthonormal vector fields along $\gamma_0(s)$ with $e_n(s)=\gamma_0'(s).$ Then in small neighborhoods $U_{x_0}$ of ${x_0}$ and $U_{y_0}$ of $y_0,$ there are mappings $x\mapsto (a_1(x),\cdots, a_n(x))$ and $y\mapsto (b_1(y),\cdots, b_n(y))$  
\[x=\exp_{x_0}\big({\sum_{1}^{n}}a_i(x)e_i(0)\big),\quad y=\exp_{y_0}\big({\sum_{1}^{n}}b_i(y)e_i(l)\big).\]
Then $\tilde d_t(x,y)$ can be defined by 
\[\tilde d_t(x,y)=\mathscr{L}_{g(t)}(\exp_{\gamma_0(s)}(\frac{l-s}{l}\sum a_i(x)e_i(s)+\frac{s}{l}\sum b_i(y)e_i(s))), s\in [0,l].\]  
Therefore we have 
\[\Psi(u(y,t),t)-\Psi(u(x,t),t)-\tilde d_t(x,y)\leq \epsilon(1+t),\]
for any $(x,y,t)\in U_{x_0}\times U_{y_0}\times [0,T]$ and with equality at $(x_0,y_0,t_0).$
Thus we can  apply the maximum principle of the parabolic version to conclude that for each $\lambda >0 $ there exist $X\in \mathcal{L}^2_s(TM_{x_0}), Y\in \mathcal{L}^2_s(TM_{y_0})$ such that 

\[(b_1,D_y  \tilde d_t(x,y)\big)\big|_{(t_0,x_0,y_0)},Y)\in \mathcal{P}^{2,+}(\Psi \circ u)({y_0},{t_0}),\]

\[(-b_2,-D_x\tilde d_t(x,y)\big)\big|_{(t_0,x_0,y_0)},X)\in \mathcal{P}^{2,-}(\Psi \circ u)(x_0,t_0),\]
and
\begin{align*}
 &b_1+b_2=\epsilon+\frac{d}{dt}\big(\tilde d_t(x,y)\big)\big|_{(t_0,x_0,y_0)},\\
 &{\left( \begin{array}{ccc}
-X &  0\\
0&Y
\end{array}
\right )} \leq H+\lambda H^2,
\end{align*}
where $H=D^2\tilde d_t(x,y)\big|_{(t_0,x_0,y_0)}.$\\
We can compute
\begin{align*}
\frac{d}{dt}\big(\tilde d_t(x,y)\big)\big|_{(t_0,x_0,y_0)}&=\int_{0}^{l}\frac{d}{dt}\big(\langle\gamma_0'(s),\gamma_0'(s) \rangle^{\frac{1}{2}}_{g(t)}\big)\Big|_{t=t_0}ds\\
&=\frac{1}{2}\int_{0}^{l}\frac{dg}{dt}\langle\gamma_0'(s),\gamma_0'(s)\rangle \Big|_{t=t_0}ds\\
&\geq -\int_{0}^{l}Ric_{t_0}(e_n(s),e_n(s))ds.
\end{align*}
Therefore, we have
\begin{equation}
b_1+b_2\geq \epsilon -\int_{0}^{l}Ric_{t_0}(e_n(s),e_n(s))ds.
\end{equation}
Note that $D_y  \tilde d_t(x,y)\big)\big|_{(t_0,x_0,y_0)}=e_n(l)$ and  $D_x \tilde d_t(x,y)\big)\big|_{(t_0,x_0,y_0)}=-e_n(0).$ By Lemma 2.4, we obtain
\begin{align*}
(b_1\varphi'(z_{y_0},t_0)+\varphi_t(z_{y_0},t_0),\varphi'(z_{y_0},t_0)e_n(l),\varphi''(z_{y_0},t_0)Y&+\varphi'(z_{y_0},t_0)e_n(l)\otimes e_n(l))\\
&\in \mathcal{P}^{2,+}(u)({y_0},{t_0}),
\end{align*}
and
\begin{align*}
(-b_2\varphi'(z_{x_0},t_0)+\varphi_t(z_{x_0},t_0),\varphi'(z_{x_0},t_0)e_n(0),\varphi''(z_{x_0},t_0)X&+\varphi'(z_{x_0},t_0)e_n(0)\otimes e_n(0))\\
&\in \mathcal{P}^{2,-}( u)({x_0},{t_0}),
\end{align*}
where $z_{x_0}=\Psi(u(x_0,t_0),t_0),z_{y_0}=\Psi(u(y_0,t_0),t_0).$
On the other hand,  since $u$ is both subsolution and supersolution of  (1.4), we have 
\begin {align*}
&\varphi_t(z_{y_0},t_0)+\varphi'(z_{y_0},t_0)b_1+q(\varphi(z_{y_0},t_0),\varphi'(z_{y_0},t_0),t_0)-\\&\tr(\varphi'(z_{y_0},t_0)A_2Y+\varphi''(z_{y_0},t_0)A_2e_n(l)\otimes e_n(l))\leq 0,
\end{align*}
and
\begin {align*}
&\varphi_t(z_{x_0},t_0)-\varphi'(z_{x_0},t_0)b_2+q(\varphi(z_{x_0},t_0),\varphi'(z_{x_0},t_0),t_0)-\\&\tr(\varphi'(z_{x_0},t_0)A_1Y+\varphi''(z_{x_0},t_0)A_1e_n(0)\otimes e_n(0))\leq 0.
\end{align*}
where 
\[ A_1={\left( \begin{array}{cccc}
\beta(t_0) &  &  &\\
& \ddots & &\\
&&\beta(t_0)&\\
&&&\alpha(\varphi(z_{x_0},t_0),\varphi'(z_{x_0},t_0),t_0)
\end{array}
\right )}, \]
\[ A_2={\left( \begin{array}{cccc}
\beta(t_0) &  &  &\\
& \ddots & &\\
&&\beta(t_0)&\\
&&&\alpha(\varphi(z_{y_0},t_0),\varphi'(z_{y_0},t_0),t_0)
\end{array}
\right )}. \]
For the inequality at $y_0,$ we have
\begin{align*}
\varphi_t(z_{y_0},t_0) -&(\alpha(\varphi(z_{y_0},t_0),\varphi'(z_{y_0},t_0),t_0))\varphi''(z_{y_0},t_0)+q(\varphi(z_{y_0},t_0),\varphi'(z_{y_0},t_0),t_0))\\
&+\varphi'(z_{y_0},t_0)\Bigg( b_1-\tr{\left( \begin{array}{ccc}
0 &  C\\
C&A_2
\end{array}
\right )}{\left( \begin{array}{ccc}
-X&  0\\
0&Y
\end{array}
\right )} \Bigg )\leq 0,
\end{align*}

where $C$ is an $n\times n$ matrix to be determined. Multiplying by $\frac{1}{\varphi'(z_{y_0},t_0)\beta(t_0)}$ gives 
\begin{align*}
& \Bigg ( \frac{\varphi_t-\varphi''\alpha(\varphi, \varphi',t)+q(\varphi,\varphi',t)}{\varphi'\beta(t)} \Bigg )\Bigg|_{({z_{y_0},t_0})}+\\
& \frac{1}{\beta(t_0)}\Bigg (b_1-\tr{\left( \begin{array}{ccc}
0 &  C\\
C&A_2
\end{array}
\right )}{\left( \begin{array}{ccc}
-X&  0\\
0&Y
\end{array}
\right )} \Bigg )\leq 0.
\end{align*}
Similarly, for the inequality at $x_0$, we get 
\begin{align*} 
&\Bigg ( \frac{\varphi_t-\varphi''\alpha(\varphi, \varphi',t)+q(\varphi,\varphi',t)}{\varphi'\beta(t)} \Bigg )\Bigg|_{({z_{x_0},t_0})}-\\
&\frac{1}{\beta(t_0)}\Bigg (b_2-\tr{\left( \begin{array}{ccc}
A_1&  0\\
0&0
\end{array}
\right )}{\left( \begin{array}{ccc}
-X &  0\\
0&Y
\end{array}
\right )} \Bigg )\geq 0.
\end{align*}
Let \[ C={\left( \begin{array}{cccc}
\beta(t_0) &  &  &\\
& \ddots & &\\
&&\beta(t_0)&\\
&&&0 
\end{array}
\right )}, \]
we obtain
\[\Bigg ( \frac{\varphi_t-\varphi''\alpha(\varphi, \varphi',t)+q(\varphi,\varphi',t)}{\varphi'\beta(t)} \Bigg )\Bigg|_{({z_{x_0},t_0})}^{({z_{y_0},t_0})}+ \frac{1}{\beta(t_0)}(b_1+b_2)-\tr \Bigg(W{\left( \begin{array}{ccc}
-X&  0\\
0&Y
\end{array}
\right )}\Bigg)\leq 0,\]
where the matrix 
\[W={\left( \begin{array}{cccc}
\ I_{n-1} &  0& I_{n-1}  &0\\
0 & \frac{\alpha(\varphi(z_{x_0},t_0),\varphi'(z_{x_0},t_0),t_0)}{\beta(t_0)} & 0&0\\
I_{n-1}&0&I_{n-1}&0\\
0&0&0&\frac{\alpha(\varphi(z_{y_0},t_0),\varphi'(z_{y_0},t_0),t_0)}{\beta(t_0)}
\end{array}
\right )}. \]
is positive semidefinite.
Since
\[{\left( \begin{array}{ccc}
-X &  0\\
0&Y
\end{array}
\right )} \leq H+\lambda H^2,\]
we get 
\[ \tr \Bigg(W{\left( \begin{array}{ccc}
-X&  0\\
0&Y
\end{array}
\right )}\Bigg)\leq \tr (WH)+\lambda \tr(WH^2).\]
Letting $\lambda\rightarrow 0,$ we have
\[ \Bigg ( \frac{\varphi_t-\varphi''\alpha(\varphi, \varphi',t)+q(\varphi,\varphi',t)}{\varphi'\beta(t)} \Bigg )\Bigg|_{({z_{x_0},t_0})}^{({z_{y_0},t_0})}+ \frac{1}{\beta(t_0)}(b_1+b_2)\leq \tr (WH).\]
Now we compute $\tr(WH)$ as following:
\begin{align*}
\tr (WH)&=\sum_{i=1}^{n-1}(D_{x_i}D_{x_i}\tilde d_t+2D_{x_i}D_{y_i}\tilde d_t+D_{y_i}D_{y_i}\tilde d_t)\big|_{(t_0,x_0,y_0)}\\
&+\frac{\alpha(\varphi(z_{x_0},t_0),\varphi'(z_{x_0},t_0),t_0)}{\beta(t_0)}D_{x_n}D_{x_n}\tilde d_t\big|_{(t_0,x_0,y_0)}\\
&+\frac{\alpha(\varphi(z_{y_0},t_0),\varphi'(z_{y_0},t_0),t_0)}{\beta(t_0)}D_{y_n}D_{y_n}\tilde d_t\big|_{(t_0,x_0,y_0)}.
\end{align*}
We know that 
\begin{align*}
&\sum_{i=1}^{n-1}(D_{x_i}D_{x_i}\tilde d_t+2D_{x_i}D_{y_i}\tilde d_t+D_{y_i}D_{y_i}\tilde d_t)\big|_{(t_0,x_0,y_0)}\\
&=\frac{d^2}{d\mu^2}\Bigg |_{\mu=0}\tilde d_{t_0}(\exp_{x_0}(\mu e_i(0)),\exp_{y_0}(\mu e_i(l)))\\
&=\frac{d^2}{d\mu^2}\Bigg |_{\mu=0}\mathscr L_{g(t_0)}(\exp_{\gamma_0(s)}(\mu e_i(s))_{s\in [0,l]}).
\end{align*}
Now we will use  the second variation formulae 
\[\frac{\partial^2}{\partial \mu^2}\Bigg |_{\mu=0}\mathscr{L}(\gamma(\mu,\cdot))=\int_{0}^{l}\Big(|\nabla_{\gamma_s}(\gamma_{\mu}^\perp)|^2-R(\gamma_s,\gamma_{\mu},\gamma_{\mu},\gamma_s)\Big)ds+\langle \gamma_s,\nabla_{\gamma_{\mu}}\gamma_{\mu}\rangle\big|^l_0,\] 
where $\gamma^{\perp}_\mu$ means the normal part of the variational vector. \\
Since $\gamma_{\mu}=e_i(s),$ we have $\nabla_{\gamma_s}\gamma_{\mu}^\perp=0$ and $\nabla_{\gamma_{\mu}}\gamma_{\mu}=0,$ then
\[\sum_{i=1}^{n-1}(D_{x_i}D_{x_i}\tilde d_t+2D_{x_i}D_{y_i}\tilde d_t+D_{y_i}D_{y_i}\tilde d_t)\big|_{(t_0,x_0,y_0)}=-\int_{0}^{l}Ric_{t_0}(e_n(s),e_n(s))ds.\]
Similarly we have
\begin{align*}
&D_{x_n}D_{x_n}\tilde d_t\big|_{(t_0,x_0,y_0)}=0,\\
&D_{y_n}D_{y_n}\tilde d_t\big|_{(t_0,x_0,y_0)}=0.
\end{align*}
In summary, we have 
\[ \tr (WH)=-\int_{0}^{l}Ric_{t_0}(e_n(s),e_n(s))ds,\]
which implies
\[\Bigg ( \frac{\varphi_t-\varphi''\alpha(\varphi, \varphi',t)+q(\varphi,\varphi',t)}{\varphi'\beta(t)} \Bigg )\Bigg|_{({z_{x_0},t_0})}^{({z_{y_0},t_0})}+ \frac{1}{\beta(t_0)}(b_1+b_2)\leq -\int_{0}^{l}Ric_{t_0}(e_n(s),e_n(s))ds,\]
Then combine this with (2.3), we obtain
\[\epsilon\leq(1-\beta(t_0))\int_{0}^{l}Ric_{t_0}(e_n(s),e_n(s))ds,\]

which gives a contradiction. Therefore we must have
\[\Psi(u(y,t),t)-\Psi(u(x,t),t)-d_t(x,y)\leq 0.\]
\end{proof}
\begin{proof}[Proof of Theorem \ref{thm:1.2}]
Let $\epsilon >0$ be arbitrary, and consider the first time $t_0>0$ and points $x_0$ and $y_0$ in $M$ at which the inequality 
\[u(y,t)-u(x,t)-2\varphi(\frac{d_t(x,y)}{2},t)-\epsilon(1+t)\leq 0\]
reaches equality. Note that if $\epsilon>0,$ then we necessarily have $y_0\neq x_0$. We replace $d_t(x,y)$ by a smooth function $\tilde d_t(x,y)$ as in the proof of Theorem  1.4 within a neighborhood of $(x_0,y_0)$. Then we will have 
\[u(y,t)-u(x,t)-2\varphi(\frac{\tilde{d}_t(x,y)}{2},t)-\epsilon(1+t)\leq 0.\]
for any $(x,y,t)\in U_{x_0}\times U_{y_0}\times [0,T]$ and with equality at $(x_0,y_0,t_0).$ Assume $l=d_{t_0}(x_0,y_0)=2s_0,$ then we apply the maximum principle of parabolic version to conclude that for each $\lambda>0,$ there exist $X\in \mathcal{L}^2_s(TM_{x_0}), Y\in \mathcal{L}^2_s(TM_{y_0})$ such that 
 \begin{align*}
&(b_1,\varphi'(s_0,t_0)D_y  \tilde d_t(x,y)\big)\big|_{(t_0,x_0,y_0)},Y)\in \mathcal{P}^{2,+}(u)({y_0},{t_0}),\\
&(-b_2,-\varphi'(s_0,t_0)D_x\tilde d_t(x,y)\big)\big|_{(t_0,x_0,y_0)},X)\in \mathcal{P}^{2,-}(u)(x_0,t_0),
\end{align*}
 and
\begin{align*}
 &b_1+b_2=\epsilon+\frac{d}{dt}(2\varphi(\frac{d_t(x,y)}{2},t))
,\\
 &{\left( \begin{array}{ccc}
-X &  0\\
0&Y
\end{array}
\right )} \leq H+\lambda H^2,
\end{align*}
where $H=D^2\psi$ and $\psi=2\varphi(\frac{\tilde{d}_t(x,y)}{2},t).$ \\
We can compute as follows:
\begin{equation}
b_1+b_2\geq \epsilon-\varphi'(s_0,t_0)\int_{0}^{l}Ric_{t_0}(e_n(s),e_n(s))ds+2\varphi_t(s_0,t_0).
\end{equation}
Since $u$ is both subsolution and supersolution of  (1.5), we have 
\[b_1\leq \tr{(A_2Y)}+q(\varphi'(s_0,t_0),t_0),\]
\[-b_2\geq \tr{(A_1X)}+q(\varphi'(s_0,t_0),t_0).\]
where 
\[ A_1=A_2={\left( \begin{array}{cccc}
\beta(\varphi'(s_0,t_0),t_0) &  &  &\\
& \ddots & &\\
&&\beta(\varphi'(s_0,t_0),t_0)&\\
&&&\alpha(\varphi'(s_0,t_0),t_0)
\end{array}
\right )}. \] 
Therefore we have
\begin {equation}
b_1\leq\tr \Bigg({\left( \begin{array}{ccc}
0&  C\\
C&A_2
\end{array}
\right )}{\left( \begin{array}{ccc}
-X &  0\\
0&Y
\end{array}
\right )}\Bigg )+q(\varphi'(s_0,t_0),t_0),
\end{equation}
\begin{equation}
-b_2\geq- \tr \Bigg({\left( \begin{array}{ccc}
A_1&  0\\
0& 0
\end{array}
\right )}{\left( \begin{array}{ccc}
-X &  0\\
0&Y
\end{array}
\right )}\Bigg )+q(\varphi'(s_0,t_0),t_0).
\end {equation}
Assume $A=A_1=A_2,$ and 
\[C={\left( \begin{array}{cccc}
\beta(\varphi'(s_0,t_0),t_0) &  &  &\\
& \ddots & &\\
&&\beta(\varphi'(s_0,t_0),t_0)&\\
&&&0
\end{array}
\right )}.\] 
Combining (2.5) with (2.6), we have
\begin{align*}
b_1+b_2&\leq \tr \Bigg({\left( \begin{array}{ccc}
A&  C\\
C&A
\end{array}
\right )}{\left( \begin{array}{ccc}
-X &  0\\
0&Y
\end{array}
\right )}\Bigg )\\
& \leq \tr \Bigg({\left( \begin{array}{ccc}
A&  C\\
C&A
\end{array}
\right )}H+\lambda{\left( \begin{array}{ccc}
A&  C\\
C&A
\end{array}
\right )}H^2\Bigg ).
\end{align*}
Multiplying by $\frac{1}{\beta (\varphi'(s_0,t_0),t_0)}$ gives 
\begin {equation}
\frac{b_1+b_2}{\beta(\varphi'(s_0,t_0),t_0)}\leq \tr(WH)+\lambda \tr (WH^2),
\end{equation}
where 
\[W={\left( \begin{array}{cccc}
\ I_{n-1} &  0& I_{n-1}  &0\\
0 & \frac{\alpha(\varphi'(s_0,t_0),t_0)}{\beta(\varphi'(s_0,t_0),t_0)} & 0&0\\
I_{n-1}&0&I_{n-1}&0\\
0&0&0& \frac{\alpha(\varphi'(s_0,t_0),t_0)}{\beta(\varphi'(s_0,t_0),t_0)}
\end{array}
\right )}\]
is positive semidefinite.\\
\begin{align*}
&\sum_{i=1}^{n-1}(D_{x_i}D_{x_i}\psi+2D_{x_i}D_{y_i}\psi+D_{y_i}D_{y_i}\psi)\big|_{(t_0,x_0,y_0)}\\
&=\sum_{i=1}^{n-1}\Bigg(\varphi'(s_0,t_0)\frac{d^2}{d\mu^2}\Bigg |_{\mu=0}\tilde d_{t_0}(\exp_{x_0}(\mu e_i(0)),\exp_{y_0}(\mu e_i(l)))\\
&+\varphi''(s_0,t_0)\frac{d}{d\mu}\Bigg |_{\mu=0}\tilde d_{t_0}(\exp_{x_0}(\mu e_i(0)),\exp_{y_0}(\mu e_i(l)))\Bigg)\\
&=\sum_{i=1}^{n-1}\Bigg(\varphi' (s_0,t_0)\frac{d^2}{d\mu^2}\Bigg |_{\mu=0}\mathscr L_{g(t_0)}(\exp_{\gamma_0(s)}(\mu e_i(s))_{s\in [0,l]})\\
&+\varphi''(s_0,t_0)\frac{d}{d\mu}\Bigg |_{\mu=0}\mathscr L_{g(t_0)}(\exp_{\gamma_0(s)}(\mu e_i(s))_{s\in [0,l]})\Bigg)\\
&=-\varphi'(s_0,t_0)\int_{0}^{l}Ric_{t_0}(e_n(s),e_n(s))ds.
\end{align*}
\begin{align*}
&(D_{x_n}D_{x_n}\psi+2D_{x_n}D_{y_n}\psi+D_{y_n}D_{y_n}\psi)\big|_{(t_0,x_0,y_0)}\\
&=\varphi'(s_0,t_0)\frac{d^2}{d\mu^2}\Bigg |_{\mu=0}\tilde d_{t_0}(\exp_{x_0}(\mu e_n(0)),\exp_{y_0}(\mu e_n(l)))\\
&+\varphi''(s_0,t_0)\frac{d}{d\mu}\Bigg |_{\mu=0}\tilde d_{t_0}(\exp_{x_0}(\mu e_n(0)),\exp_{y_0}(\mu e_n(l)))\\
&=\varphi' (s_0,t_0)\frac{d^2}{d\mu^2}\Bigg |_{\mu=0}\mathscr L_{g(t_0)}(\exp_{\gamma_0(s)}(\mu e_n(s))_{s\in [0,l]})\\
&+\varphi''(s_0,t_0)\frac{d}{d\mu}\Bigg |_{\mu=0}\mathscr L_{g(t_0)}(\exp_{\gamma_0(s)}(\mu e_n(s))_{s\in [0,l]})\\
&=2\varphi''(s_0,t_0).
\end{align*}
In summary, we have
\[\tr(WH)=2\varphi''(s_0,t_0)\frac{\alpha(\varphi'(s_0,t_0),t_0)}{\beta(\varphi'(s_0,t_0),t_0)}-\varphi'(s_0,t_0)\int_{0}^{l}Ric_{t_0}(e_n(s),e_n(s))ds.\]
Substitute this to (2.7) and let $\lambda \rightarrow 0 ,$we have
\[\frac{b_1+b_2}{\beta(\varphi'(s_0,t_0),t_0)}\leq 2\varphi''(s_0,t_0)\frac{\alpha(\varphi'(s_0,t_0),t_0)}{\beta(\varphi'(s_0,t_0),t_0)}-\varphi'(s_0,t_0)\int_{0}^{l}Ric_{t_0}(e_n(s),e_n(s))ds.
\]
Combining this equation with (2.4), we have
\begin{align*}
&\epsilon\leq -2\varphi_t(s_0,t_0)+2\alpha(\varphi'(s_0,t_0),t_0) \varphi''(s_0,t_0)\\
&+\varphi'(s_0,t_0)(1-\beta(\varphi'(s_0,t_0),t_0))\int_{0}^{2s_0}Ric_{t_0}(e_n(s),e_n(s))ds,
\end{align*}
which gets a contradiction. We must have
\[u(y,t)-u(x,t)-2\varphi(\frac{d_t(x,y)}{2},t)-\epsilon(1+t)\leq 0.\]
\end{proof}


\renewcommand{\abstractname}{Acknowledgements}
\begin{abstract}
I would like to show my deepest gratitude to my
supervisor, Prof. Li, Jiayu, who has provided me with valuable guidance in every stage of the writing of this
paper.
\end{abstract}
 ~\\

\end{document}